\documentclass[a4paper,12pt,oneside]{amsart}
\pdfoutput=1
\usepackage{microtype}
\usepackage{amssymb,latexsym}
\usepackage[margin=3.25cm]{geometry}
\usepackage{mathtools}
\usepackage{amsthm}
\usepackage{enumitem}
\usepackage{mathrsfs}
\usepackage{nccmath}
\usepackage{mleftright}
\usepackage[dvipsnames]{xcolor}
\usepackage{hyperref}
\hypersetup{
    colorlinks=true,
    linkcolor=black, 
    citecolor=black} 

\DeclareMathOperator{\Ima}{Im}

\newcommand{\diff}{\mathop{}\!d}
\newcommand{\grad}{\mathop{}\!\nabla}

\theoremstyle{plain}
\newtheorem{theorem}{Theorem}[section]
\newtheorem{proposition}[theorem]{Proposition}
\newtheorem{lemma}[theorem]{Lemma}
\newtheorem{corollary}[theorem]{Corollary}
\theoremstyle{remark}
\newtheorem{remark}[theorem]{Remark}
\newtheorem*{remark*}{Remark}
\newtheorem*{remarks*}{Remarks}
\theoremstyle{definition}

\newtheorem{definition}[theorem]{Definition}
\newtheorem*{notation*}{Notation}

\begin{document}
\title[Curvature-adapted submanifolds of Lie groups]{Curvature-adapted submanifolds\\of semi-Riemannian groups}

\author{Margarida Camarinha}
\address{CMUC\\
University of Coimbra\\
Department of Mathematics\\
3000-143 Coimbra\\
Portugal}
\email{mmlsc@mat.uc.pt}
\thanks{This work was partially supported by the Centre for Mathematics of the University of Coimbra -- UIDB/00324/2020, funded by the Portuguese Government through FCT/MCTES.\endgraf
The second author was partially supported by Austrian Science Fund (FWF) project F~77 (SFB ``Advanced Computational Design'').}
\author{Matteo Raffaelli}
\address{CMUC\\
University of Coimbra\\
Department of Mathematics\\
3000-143 Coimbra\\
Portugal}
\curraddr{Institute of Discrete Mathematics and Geometry\\
TU Wien\\
Wiedner Hauptstra{\ss}e 8-10/104\\
1040 Vienna\\
Austria}
\email{matteo.raffaelli@tuwien.ac.at}

\date{May 12, 2023}

\begin{abstract}
We study semi-Riemannian submanifolds of arbitrary codimension in a Lie group $G$ equipped with a bi-invariant metric. In particular, we show that, if the normal bundle of $M \subset G$ is closed under the Lie bracket, then any normal Jacobi operator $K$ of $M$ equals the square of the associated invariant shape operator $\alpha$. This permits to understand curvature adaptedness to $G$ geometrically, in terms of left translations. For example, in the case where $M$ is a Riemannian hypersurface, our main result states that the normal Jacobi operator commutes with the ordinary shape operator precisely when the left-invariant extension of each of its eigenspaces has first-order tangency with $M$ along all the others. As a further consequence of the equality $K = \alpha^{2}$, we obtain a new case-independent proof of a well-known fact: every three-dimensional Lie group equipped with a bi-invariant semi-Riemannian metric has constant curvature.
\end{abstract}
\maketitle

\tableofcontents

\section{Introduction and main result} \label{IntroductionMainResult}
Given a Riemannian manifold, it is natural to study submanifolds whose geometry is somehow adapted to that of the ambient space. This idea led to the concept of \emph{curvature-adapted} submanifold~\cite{datri1979, berndt1991, berndt1992-2}; a concept that, since its introduction, has attracted the interest of many geometers.

Here we define curvature adaptedness in a slightly more general setting, that of a \emph{semi-Riemannian} manifold $(Q,g)\equiv Q$.

Let $M$ be an $m$-dimensional semi-Riemannian submanifold of $Q$, let $N^{1}M$ be its unit normal bundle, and let $R$ be the ambient curvature tensor. For $(p,\eta)\equiv \eta  \in N^{1}_{p}M$, the \textit{normal Jacobi operator}
\begin{align*}
K_{\eta} \equiv K \colon T_{p}M &\to T_{p}Q \\
v &\mapsto R(\eta,v)\eta
\end{align*}
\textit{of $M$ \textup{(}with respect to $\eta$\textup{)}} measures the curvature of the ambient manifold along $\eta$. On the other hand, denoting by $N$ a unit normal local extension of $\eta$ along $M$, and by $\nabla$ the Levi-Civita connection of $Q$, the \textit{shape operator}
\begin{align*}
A_{\eta} \equiv A \colon T_{p}M &\to T_{p}M \\
v &\mapsto \pi^{\top} \grad_{v}N
\end{align*}
\textit{of $M$ \textup{(}with respect to $\eta$\textup{)}} describes the curvature of $M$ as a submanifold of $Q$; here $\pi^{\top}$ denotes orthogonal projection onto $T_{p}M$.

\begin{definition}\label{curvatureAdaptednessDEF}
One says that $M$ is \textit{curvature adapted \textup{(}to $Q$\textup{)} at a point $p\in M$} if, for every $\eta \in N^{1}_{p}M$,
\begin{enumerate}[label=(\arabic*)]
\item \label{cond1} The normal Jacobi operator leaves $T_{p}M$ invariant, i.e., $K(T_{p}M) \subset T_{p}M$;
\item \label{cond2} The operators $A$ and $K$ commute, i.e., $K \circ A = A \circ K$.
\end{enumerate}

Consequently, one calls $M$ \textit{curvature adapted} if it is curvature adapted at $p$ for all $p \in M$.
\end{definition}

\begin{remark}
Condition~\ref{cond1} in Definition~\ref{curvatureAdaptednessDEF} is always satisfied for hypersurfaces.
\end{remark}

\begin{remark}
Both $A$ and $K$ are self-adjoint with respect to the induced semi-Riemannian metric. Thus, if these operators are also diagonalizable, then $M$ is curvature adapted at $p$ precisely when they share a common orthonormal basis of eigenvectors (Lemma~\ref{OrthogonalEigendecompositionLemma}). Recall that diagonalizability is automatic if the metric is positive definite.
\end{remark}

\begin{remark}
In general, it could also be interesting to consider a weaker notion of curvature adaptedness, obtained by replacing the operator $K$ with $\pi^{\top}K$ in the definition. This modification would make any totally umbilic submanifold necessarily curvature adapted.
\end{remark}

It is easy to see (e.g., using \cite[Lemma~2]{graves1978}) that every semi-Riemannian submanifold of a real semi-Riemannian space form is curvature adapted. However, for other ambient spaces, the definition is restrictive. For example, if $Q$ is an $(m+1)$-dimensional nonflat complex space form with complex structure $J$, then $A$ and $K$ commute precisely when $-J\eta$ is an eigenvector of $A$. Also, if $Q$ is an $(m+1)$-dimensional nonflat quaternionic space form with quaternionic structure $\mathfrak{I}$, then $A$ and $K$ commute precisely when the maximal subspace of $T_{p}M$ invariant under $\mathfrak{I}$ is also invariant under $A$; see \cite{berndt1991} and \cite[sec.~9.8]{cecil2015}.

In a symmetric space of nonconstant curvature, the situation is more involved, yet many interesting results have been obtained. Among others (see for example \cite{koike2005, murphy2012, koike2014-1, koike2014-2}), the most important is arguably Gray's theorem~\cite[Theorem~6.14]{gray2004}, which states that any tubular hypersurface around a curvature-adapted submanifold is itself curvature adapted.

Gray's theorem has been further generalized to the family of Riemannian manifolds such that, for every geodesic $\gamma$, the Jacobi operator $R(\dot{\gamma},\cdot)\dot{\gamma}$ is diagonalizable by a \emph{parallel} orthonormal frame field along $\gamma$. It turns out that, in such spaces, the classification of the curvature-adapted submanifolds is fully determined by that of the curvature-adapted hypersurfaces~\cite{berndt1992, berndt1992-2}.

In this article, we shall examine the case where $Q$ is a \textit{semi-Riemannian group}, i.e., a Lie group $G \equiv (G, \langle \cdot{,} \cdot \rangle)$ equipped with a bi-invariant semi-Riemannian metric $\langle \cdot{,} \cdot \rangle$. In particular, we will focus our attention on the class of (semi-Riemannian) submanifolds of $G$ having \emph{closed} normal bundle; here by \emph{closed} we mean that each normal space of $M$ corresponds, under the group's left action, to a subspace of $\mathfrak{g}$ that is closed under the Lie bracket, i.e., to a Lie subalgebra (Definition~\ref{ClosedNormalBundleDEF}).

Examples of Lie groups $G$ abounds~\cite{ovando2016}: every semisimple Lie group can be furnished with a bi-invariant metric, and every compact Lie group admits one that is Riemannian. Note that, if $G$ is absolutely simple, then any bi-invariant metric on $G$ is a scalar multiple of the Killing form of $\mathfrak{g}$. 

In order to explain our main result, we first set up some notation. Provided that $K$ is diagonalizable, let $(e_{1}, \dotsc, e_{m})$ be an orthonormal basis of eigenvectors of $K$, that is, a basis of $T_{p}M$ such that $\lvert \langle e_{j}, e_{h} \rangle\rvert =\delta_{jh}$ and $K(e_{j}) = \lambda_{j} e_{j}$ for all $j,h = 1,\dotsc,m$; for each $j$, let $e_{j}^{\mathrm{L}}$ denote the left-invariant extension of $e_{j}$.

\begin{theorem} \label{TH1}
Let $M$ be a semi-Riemannian submanifolds of a semi-Riemannian group $G$. Assume that $K$ is diagonalizable and that the normal space of $M$ at $p$ is closed under the Lie bracket. Then the following statements are equivalent:
\begin{enumerate}[font=\upshape, label=(\roman*)]
\item $A$ and $K$ commute.
\item \label{item2} If $\lambda_{j}$ and $\lambda_{h}$ are distinct eigenvalues, then $e_{j}(\langle N, e_{h}^{\mathrm{L}} \rangle) = 0$.
\end{enumerate}
Moreover, if the induced metric is positive definite at $p$, then $e_{j}(\langle N, e_{h}^{\mathrm{L}} \rangle) = e_{h}(\langle N, e_{j}^{\mathrm{L}} \rangle)$ whenever $e_{j}$ and $e_{h}$ are in distinct eigenspaces.
\end{theorem}

\begin{remark}
For all $j =1,\dotsc,m$, we have $\lambda_{j} = -\mathrm{sec}(e_{j},\eta) \leq 0$; see section~\ref{PRE}.
\end{remark}

Roughly speaking, Theorem~\ref{TH1}\ref{item2} expresses the condition that the left-invariant extension of each eigenspace of $K$ is ``orthogonal to first order" to $N$ (at $p$) along all other eigenspaces; in other words, along any curve that starts at $p$ and is tangent to a different eigenspace. 

In particular, in the case of hypersurfaces, the condition on the normal space is automatically satisfied. Specializing the theorem to that case, we obtain the result below.
\begin{corollary}
If $M$ is a hypersurface and $K$ is diagonalizable, then the following statements are equivalent:
\begin{enumerate}[font=\upshape, label=(\roman*)]
\item $A$ and $K$ commute.
\item The left-invariant extension of each eigenspace of $K$ has first-order tangency with $M$ along all other eigenspaces.
\end{enumerate}
\end{corollary}

The main significance of Theorem~\ref{TH1} lies in the fact that it permits to understand curvature adaptedness to $G$ geometrically, in terms of left translations. More precisely, it reveals that, in order for a generic submanifold (with closed normal bundle) to be curvature adapted, its tangent bundle needs to behave reasonably well under left translations. Note that the tangent bundle of a Lie subgroup is fully left-invariant; conversely, if $M$ is a closed, connected submanifold that contains the identity and has left-invariant tangent bundle, then it is a Lie subgroup.

The basic fact that allows us to prove Theorem~\ref{TH1} is that the shape operator of $M\subset H$ with respect to $\eta$, being $H$ any Lie group with a left-invariant semi-Riemannian metric, decomposes as the sum of two terms~\cite{ripoll1991}: an \emph{invariant shape operator}, which depends only on $\eta$, $T_{p}M$ and $H$; plus a second term, here denoted by $\mathcal{W}$, that is closely related to the Gauss map of $M$; see section~\ref{ISOP} for details. In particular, if the metric is bi-invariant and the normal bundle is closed, then the invariant shape operator commutes with $K$ (Proposition~\ref{PROP9}), and so, by linearity, commutativity of $A$ and $K$ reduces to that of $\mathcal{W}$ and $K$.

In fact, if $\dim M = 2$ or $M$ is Riemannian, then the nonzero eigenvalues of $K$ have even multiplicities (Corollary~\ref{COR11}), which leads us to the following conclusions.
\begin{proposition}\label{PROP4.0}
Every two-dimensional surface with closed normal bundle is curvature adapted to $G$.
\end{proposition}

\begin{proposition}\label{PROP4}
If $M$ is a three-dimensional Riemannian submanifold of $G$ with closed normal bundle, and if $K \neq 0$ for all $\eta \in N^{1}_{p}M$, then the following statements are equivalent:
\begin{enumerate}[font=\upshape, label=(\roman*)]
\item $M$ is curvature adapted in a neighborhood $U$ of $p$.
\item For all $\eta \in N^{1}U$, the $0$-eigenvector of $K$ is an eigenvector of $A$.
\end{enumerate}
\end{proposition}

Proposition~\ref{PROP4.0} implies that, if $\dim G =3$, then every Riemannian surface in $G$ is curvature adapted.  This is by no means surprising, because every three-dimensional semi-Riemannian group has constant curvature. Indeed, it is well known that the Lie algebra of such a Lie group is either abelian or isomorphic to $\mathfrak{sl}(2,\mathbb{R})$ or $\mathfrak{su}(2)$~\cite{ovando2016}; in each of the latter two cases, a direct computation would reveal that the curvature of the Killing form (which, up to scaling, coincides with the metric) vanishes. Alternatively, since the Ricci curvature of $G$ is proportional to the Killing form, the statement follows from the fact that every three-dimensional Einstein manifold has constant curvature~\cite[p.~49]{besse2008}; see Remark~\ref{alternativeProof}.

On the other hand, applying classical results of Cartan and Dajczer--Nomizu (which, for the reader's convenience, are included in section~\ref{PRE}), we can prove the following statement.

%
\begin{lemma}
 Suppose that $Q$ is a Riemannian \textup{(}resp., Lorentzian\textup{)} manifold of dimension at least three. If, for some $p\in Q$ and each $x \in T_{p}Q$ such that $\langle x, x \rangle=1$ \textup{(}resp., $-1$\textup{)}, the map from $x^{\perp}$ to $x^{\perp}$ defined by $y \mapsto R(x,y)x$ is a multiple of the identity, then $Q$ has constant curvature.
 \end{lemma}

\begin{proof}
Assume the hypothesis of the lemma. If $x,y,z$ are orthonormal vectors in $T_{p}Q$ (resp., if $x,y,z$ are orthonormal vectors in $T_{p}Q$ and $\langle x, x \rangle =-1$), then $\langle R(x,y)z,x\rangle = -\langle R(x,y)x,z\rangle=0$. Applying \cite[Lemma~1.17]{dajczer2019} (resp., \cite[Theorem~1a]{dajczer1980}) and Schur's lemma~\cite[p.~96, Exercise 21(b)]{oneill1983}, the claim follows.
\end{proof}

Corollary~\ref{COR11} thus yields a \emph{case-independent} proof that, in dimension three, every metric Lie group has constant curvature.

\begin{remark}\label{alternativeProof}
An alternative case-independent proof of the same fact may be sketched as follows. First, by Levi decomposition and the absence of simple Lie algebras of dimension one and two, observe that a three-dimensional Lie algebra is either solvable or simple. However, if a nonabelian solvable Lie algebra of dimension three admits an ad-invariant bilinear form, then such form is necessarily degenerate~\cite[Proposition~2.3]{delbarco2014}; in other words, a three-dimensional nonabelian Lie algebra admitting an ad-invariant metric is (absolutely) simple.

Suppose, thus, that $G$ is absolutely simple. Then any semi-Riemannian metric on $G$ is a scalar multiple of the Killing form of $\mathfrak{g}$, implying that $G$ is an Einstein manifold. Hence, in dimension three, it has constant curvature~\cite[p.~49]{besse2008}.
\end{remark}

The remainder of the paper is organized as follows.
In the next section we briefly review some background material. In section~\ref{ISOP} we introduce the invariant shape operator and examine its properties. In section~\ref{PROOF} we then prove Theorem~\ref{TH1} and Proposition~\ref{PROP4}. We conclude with Appendix~\ref{APPA}, where, for the sake of illustration, we give a direct proof that condition \ref{item2} in Theorem~\ref{TH1} holds whenever $p$ is an umbilical point of $M$.

Some final remarks about notation:
\begin{enumerate}
\item The indices $j,h,i$ satisfy $j,h \in \{1, \dotsc, m\}$ and $i \in \{1, \dotsc, m+n\}$; note that we always use Einstein summation convention.
\item For $x \in T_{p}H$, we denote the left-invariant extension of $x$ by $x^{\mathrm{L}}$.
\end{enumerate}

\section{Preliminaries} \label{PRE}
Here we recall some basic results that are used throughout the paper; see e.g.\ \cite{milnor1976, lee2018, oneill1983} for further details about semi-Riemannian geometry and metric Lie groups.

\subsection{Semi-Euclidean vector spaces}
Let $V$ be a real vector space, of finite dimension $d$, equipped with a nondegenerate symmetric bilinear form $f$, and let $L$ be an endomorphism of $V$ that is self-adjoint with respect to $f$, i.e., such that $f(L(x), y) = f(x,L(y))$ for all $x,y\in V$. Recall that a basis $v_{1}, \dotsc, v_{d}$ of $V$ is called \textit{orthonormal} if $\vert f(v_{j},v_{h}) \rvert= \delta_{jh}$ for all $j,h =1, \dotsc, d$.

It is well known that, if $f$ is positive definite, then $L$ is diagonalizable by an orthonormal basis of eigenvectors of $L$. Moreover, two self-adjoint endomorphisms of $V$ commute if and only if they share a common orthonormal basis of eigenvectors.

While it is not possible to fully extend these classic results beyond the positive definite case, something interesting can still be said.

\begin{lemma} \label{OrthogonalEigendecompositionLemma}
	If $L$ is diagonalizable, then there exists an orthonormal basis of eigenvectors. Moreover, any two  diagonalizable self-adjoint endomorphisms of $V$ commute if and only if they share a common orthonormal basis of eigenvectors. 
\end{lemma}
\begin{proof}
Assume that $L$ is diagonalizable. Since $V$ is the direct sum of \emph{mutually orthogonal} eigenspaces of $L$, each eigenspace must be nondegenerate, and so it has an orthonormal basis; by concatenating these bases, we obtain an orthonormal basis of $V$.

As for the second statement, one direction is obvious; for the other, it suffices to note that any two commuting linear maps on $V$ preserve each other's eigenspaces.
\end{proof}

\subsection{Semi-Riemannian geometry}
Let $(Q,g)$ be a semi-Riemannian manifold, and let $\nabla$ be its Levi-Civita connection. The curvature endomorphism $R \colon \mathfrak{X}(Q)^{3} \to \mathfrak{X}(Q)$ of $(Q,g)$ is the $(1,3)$-tensor field on $Q$ defined by
\begin{equation*}
	R(X,Y)Z= \nabla_{X}\nabla_{Y}Z-\nabla_{Y}\nabla_{X}Z-\nabla_{[X,Y]}Z.
\end{equation*}
(Caution: some authors define the curvature endomorphism as the negative of ours.)

Being $R$ a tensor, if $x,y,z$ are vectors in $T_{p}Q$, then the value $R(x,y)z$ is independent of the extension of $x,y,z$ and thus well-defined. 

Suppose that $x \in T_{p}Q$ has unit length, i.e., that $\lvert g(x,x) \rvert =1$. Then the \textit{Jacobi operator of $(Q,g)$ with respect to $x$} is the linear map
\begin{align*}
	K_{x} \colon x^{\perp} &\to x^{\perp}\\
	y &\mapsto R(x,y)x.
\end{align*}
Clearly, by the symmetry by pairs of the $(0,4)$-curvature tensor (obtained by lowering the last index of $R$), the operator $K_{x}$ is self-adjoint with respect to $g$.

Next, suppose that $x,y \in T_{p}Q$ are orthonormal. Then the sectional curvature $\mathrm{sec}(x,y)$ of the nondegenerate plane spanned by $x$ and $y$ is given by the formula
\begin{equation*}
	\mathrm{sec}(x,y)= g(R(x,y)y,x)=-g(K_{x}(y),y),
\end{equation*}
where the last equality follows from the skew-symmetry of the $(0,4)$-curvature tensor.

In the Riemannian setting, an important criterion for discerning whether a manifold has constant sectional curvature is provided by the following lemma.

\begin{lemma}[{\cite{cartan1946}, \cite[Lemma~1.17]{dajczer2019}}] \label{RiemLM}
	Suppose that $Q$ is a Riemannian manifold, and that $\dim Q \geq 3$. If, at some point $p \in Q$, the curvature tensor satisfies $g(R(x,y)z, x)=0$ whenever $x,y,z$ are orthonormal, then all sectional curvatures of $Q$ at $p$ are equal.
\end{lemma}

Several generalizations of Lemma~\ref{RiemLM} to the semi-Riemannian setting have appeared~\cite{graves1978, dajczer1980, nomizu1983}. Below we recall the one that is most relevant to our discussion.

\begin{definition}
	Let $x,y \in T_{p}Q$. We say that the pair $(x,y)$ is \textit{orthonormal of signature $(-,+)$} if $g(x,x) =-1$, $g(y,y) =1$, and $g(x,y) =0$.
\end{definition}

\begin{theorem}[{\cite[Theorem~1a]{dajczer1980}}]
	Suppose that $\dim Q \geq 3$. If, at some point $p \in Q$, the curvature tensor satisfies $g(R(x,y)z,x) =0$ whenever $(x,y)$ is orthonormal of signature $(-,+)$ and $g(x,z) = g(y,z) =0$, then all nondegenerate two-planes have the same sectional curvature.
\end{theorem}

\subsection{Semi-Riemannian groups}
Let $G$ be a Lie group equipped with a left- and right-invariant (i.e., bi-invariant) semi-Riemannian metric $\langle \cdot{,}\cdot \rangle$, and let $\mathfrak{g}$ be its Lie algebra, that is, the Lie algebra of left-invariant vector fields on $G$. As customary, we identify $\mathfrak{g}$ with the tangent space $T_{e}G$ of $G$ at the identity $e$.

Suppose that $X, Y, Z \in \mathfrak{X}(G)$ are left-invariant, i.e., that $X,Y,Z \in \mathfrak{g}$. Then the Levi-Civita connection is given by
\begin{equation}\label{EQ2}
	\nabla_{X}Y = -\nabla_{Y}X= \frac{1}{2}[X,Y]
\end{equation}
and the curvature endomorphism by
\begin{equation} \label{curvatureEQ}
	R(X,Y)Z = \frac{1}{4}[Z,[X,Y]].
\end{equation}
In addition, the following equality holds:
\begin{equation}\label{EQ3}
	\langle [X,Y],Z \rangle = \langle X,[Y,Z] \rangle.
\end{equation}

Suppose that $x,y \in T_{p}G$ are orthonormal; let $x^{\mathrm{L}},y^{\mathrm{L}}$ be their left-invariant extensions. The sectional curvature of the the two-plane spanned by $x$ and $y$ may be computed by
\begin{equation*}
	\mathrm{sec}(x,y) =\frac{1}{4} \langle [x^{\mathrm{L}},y^{\mathrm{L}}], [x^{\mathrm{L}},y^{\mathrm{L}}] \rangle.
\end{equation*}
Note that $\mathrm{sec}(x,y) \geq 0$, with equality if and only if $[x^{\mathrm{L}},y^{\mathrm{L}}]=0$.

Now, let $M$ be a semi-Riemannian submanifold of $G$.

\begin{definition}\label{ClosedNormalBundleDEF}
The normal space $N_{p}M$ is said to be \textit{closed \textup{(}under the Lie bracket\textup{)}} if $dL_{p^{-1}}(N_{p}M)$ is a Lie subalgebra of $\mathfrak{g}$. Consequently, one calls the normal bundle of $M$ \textit{closed \textup{(}under the Lie bracket\textup{)}} if every normal space is closed.
\end{definition}

It is clear that $N_{p}M$ is closed exactly when $\exp(N_{p}M)$ is contained in a Lie subgroup of $G$.

\begin{remark}
Following \cite{terng1995}, the normal space $N_{p}M$ is called \textit{abelian} if $\exp(N_{p}M)$ is contained in a totally geodesic, flat submanifold of $G$. It is easy to see that $N_{p}M$ is abelian if and only if $dL_{p^{-1}}(N_{p}M)$ is an abelian subalgebra of $\mathfrak{g}$.
\end{remark}

\section{The invariant shape operator} \label{ISOP}
In this section we consider the general case of an orientable semi-Riemannian submanifold $M$ of a Lie group $H$ equipped with a left-invariant metric $\langle \cdot {,}\cdot\rangle$. Given a unit normal vector $\eta$ of $M$ at $p$, the \textit{invariant shape operator of $M$ \textup{(}with respect to $\eta$\textup{)}} is the map
\begin{align*}
\alpha \colon T_{p}M &\to T_{p}M \\
v &\mapsto \pi^{\top}\grad_{v}\eta^{\mathrm{L}},
\end{align*}
where, as usual, $\pi^{\top}$ is the orthogonal projection onto $T_{p}M$ and $\eta^{\mathrm{L}}$ the left-invariant extension of $\eta$.

The significance of the invariant shape operator lies in the fact that it represents the deviation of the ordinary shape operator from the differential of the Gauss map of $M$, as we now explain.

Let $N^{1}M$ be the unit normal bundle of $M$, and let $\mathbb{S}^{m+n-1}_{e}$ be the unit sphere inside the Lie algebra $\mathfrak{h}$ of $H$. The \textit{Gauss map of $M$} is the map
\begin{align*}
\mathcal{G}\colon N^{1}M &\to \mathbb{S}^{m+n-1}_{e}\\
(p,\eta) &\mapsto \diff\mleft(L_{p^{-1}}\mright)(\eta);
\end{align*}
here $L_{p^{-1}} \colon T_{p}H \to \mathfrak{h}$ denotes left translation by $p^{-1}$.

Let $N$ be a unit normal vector field along $M$ such that $N_{p} = \eta$, and consider the map $\bar{\mathcal{G}} = \mathcal{G} \circ N$. Its differential at $p$ is a linear map $T_{p}M \to \mathcal{G}(p,\eta)^{\perp}$. Thus, since $\diff(L_{p^{-1}})$ takes $\eta$ to $\mathcal{G}(p,\eta)$ and is an isometry, it follows that
\begin{equation*}
\mathcal{W} = \pi^{\top} \circ \diff\mleft(L_{p}\mright) \circ \diff\bar{\mathcal{G}}
\end{equation*}
is an endomorphism of $T_{p}M$.

\begin{remark}
The Gauss map of a hypersurface in a metric Lie group was first defined by Ripoll in \cite{ripoll1991}. It is worth pointing out that our definition can be extended to the case where the ambient manifold is parallelizable~\cite{ripoll1993}, or even just Killing-parallelizable~\cite{fornari2004}.
\end{remark}

Clearly, if $H=\mathbb{R}^{m+1}$, then $\mathcal{G}$ is the classical Gauss map of $M$, whereas $\mathcal{W}$ its shape operator. In our setting, the following result holds.

\begin{proposition}[cf.\ {\cite[p.~769]{ripoll1993}}]\label{PROP6}
For any $v \in T_{p}M$,
\begin{equation} \label{EQ4}
A(v)  = \alpha(v) + \mathcal{W}(v).
\end{equation}
\end{proposition}
\begin{proof}
Let $(b_{1}, \dotsc, b_{m+n})$ be an orthonormal basis of $T_{p}H$ such that $b_{1},\dotsc,b_{m} \in T_{p}M$ and $b_{m+n}=\eta$. For each $i$, let $b^{\mathrm{L}}_{i}$ be the left-invariant extension of $b_{i}$, so that $(b^{\mathrm{L}}_{1}, \dotsc, b^{\mathrm{L}}_{m+n}= \eta^{\mathrm{L}})$ is an orthonormal frame for $H$ (and a basis of $T_{e}H$).

If $q \in M$---writing $N^{i}$ as a shorthand for $\langle N, b^{\mathrm{L}}_{i} \rangle$---then
\begin{equation*}
\bar{\mathcal{G}}(q) = \diff\mleft(L_{q^{-1}}\mright)\mleft(N_{q}\mright) = \diff\mleft(L_{q^{-1}}\mright) \bigl( N^{i}(q) \mleft. b^{\mathrm{L}}_{i} \mright\rvert_{q}\bigr)=N^{i}(q) \mleft. b^{\mathrm{L}}_{i}\mright \rvert_{e}.
\end{equation*}
Thus, if $v \in T_{p}M$, then
\begin{equation*}
\diff\bar{\mathcal{G}}(v) = dN^{i}(v) \mleft. b^{\mathrm{L}}_{i} \mright \rvert_{e} = v (N^{i}) \mleft. b^{\mathrm{L}}_{i} \mright \rvert_{e}.
\end{equation*}
Since
\begin{equation*}
\diff\mleft(L_{p}\mright)(\diff\bar{\mathcal{G}}(v))=v (N^{i} ) b_{i},
\end{equation*}
it follows that
\begin{equation}\label{EQ5}
\pi^{\top}\diff\mleft(L_{p}\mright)(\diff\bar{\mathcal{G}}(v))=v(N^{j}) b_{j}.
\end{equation}

On the other hand,
\begin{align*}
A(v) &= \pi^{\top}\grad_{v}N^{i}b^{\mathrm{L}}_{i}\\
&=\pi^{\top} \mleft( N^{i}(p)\grad_{v}b^{\mathrm{L}}_{i} + v(N^{i}) b_{i}\mright).
\end{align*}
Since, by construction, $N^{1}(p)= \dotsb =N^{m+n-1}(p) =0$ and $N^{m+n}(p)=1$, we have
\begin{equation} \label{EQ6}
A(v) =\alpha(v) + v(N^{j}) b_{j},
\end{equation}
which, together with \eqref{EQ5}, gives \eqref{EQ4}.
\end{proof}

\begin{remark}
Proposition~\ref{PROP6} shows that $\mathcal{W}$ does not depend on the particular choice of normal vector field $N$ but only on its value at $p$.
\end{remark}

\begin{remark}
Using equation \eqref{EQ6}, it is not difficult to see that statement \ref{item2} in Theorem~\ref{TH1} is nothing but the coordinate expression, with respect to the frame $(e^{\mathrm{L}}_{1}, \dotsc, e^{\mathrm{L}}_{m})$, of the condition
\begin{equation} \label{EQ7}
\pi_{j} \Ima \mathcal{W} \rvert_{\Lambda_{h}} = 0 \quad \text{for all $j,h = 1, \dotsc, m$ such that $\lambda_{j} \neq \lambda_{h}$},
\end{equation}
where $\Lambda_{h}$ is the eigenspace of $K$ corresponding to the eigenvalue $\lambda_{h}$, and where $\pi_{j}$ is the orthogonal projection onto $\Lambda_{j}$. Note that \eqref{EQ7} holds if and only if $\mathcal{W}$ leaves the eigenspaces of $K$ invariant.
\end{remark}
%

A useful property of the invariant shape operator, which is crucial in proving Theorem~\ref{TH1}, is contained in the following proposition.

\begin{proposition}\label{PROP9}
If $\langle \cdot{,} \cdot \rangle$ is bi-invariant and the normal space $N_{p}M$ is closed, then
\begin{enumerate}[font=\upshape]
\item $K=\alpha\circ\alpha$, and so $\alpha$ and $K$ commute;
\item $K$ leaves $T_{p}M$ invariant.
\end{enumerate}
\end{proposition}

The proof will be based on a lemma.

\begin{lemma} \label{LM10}
Under the hypotheses of Proposition~\ref{PROP9}, $\alpha(v)=\nabla_{v}\eta^{\mathrm{L}}$.
\end{lemma}
\begin{proof}
Let $\xi$ be a unit normal vector at $p$. Since $N_{p}M$ is closed,
\begin{equation*}
[\eta^{\mathrm{L}},\xi^{\mathrm{L}}] \in \diff L_{p^{-1}}(N_{p}M),
\end{equation*}
under the usual identification of $\mathfrak{g}$ with $T_{e}G$. By the bi-invariance of the metric, it follows that
\begin{equation*}
\mleft\langle \mleft[ v^{\mathrm{L}}, \eta^{\mathrm{L}} \mright], \xi^{\mathrm{L}} \mright\rangle = \mleft\langle v^{\mathrm{L}}, \mleft[\eta^{\mathrm{L}}, \xi^{\mathrm{L}} \mright]\mright\rangle=0,
\end{equation*}
which implies $1/2[v^{\mathrm{L}},\eta^{\mathrm{L}}]_{p}=\nabla_{v}\eta^{\mathrm{L}} \in T_{p}M$, and so $\alpha(v)= \pi^{\top}\grad_{v}\eta^{\mathrm{L}}= \nabla_{v}\eta^{\mathrm{L}}$.
\end{proof}

\begin{remark}
The converse of Lemma~\ref{LM10} holds: if the metric is bi-invariant and $\alpha(v) = \nabla_{v}\eta^{\mathrm{L}}$ for all $(v,\eta) \in T_{p}M\times N_{p}M$, then $N_{p}M$ is closed. In other words, given a nondegenerate subspace $S$ of $\mathfrak{g}$, the orthogonal complement $S^{\perp}$ of $S$ is closed under the Lie bracket if and only if $[S, S^{\perp}] \subset S$.
\end{remark}

\begin{proof}[Proof of Proposition~\ref{PROP9}]
Clearly, being the second assertion in the proposition a direct consequence of the first, we only need to prove the latter.

Let $v \in T_{p}M$. Since $K$ is tensorial, the value $K(v)$ may be computed in terms of the left-invariant extensions $v^{\mathrm{L}}$ and $\eta^{\mathrm{L}}$ of $v$ and $\eta$:
\begin{equation*}
K(v) = R(\eta^{\mathrm{L}}, v^{\mathrm{L}})\eta^{\mathrm{L}}.
\end{equation*}

Assume that the metric is bi-invariant. Then, using \eqref{EQ2} and \eqref{curvatureEQ}, we have
\begin{align*}
K(v) &= \frac{1}{4} \mleft[ \eta^{\mathrm{L}}, \mleft[\eta^{\mathrm{L}}, v^{\mathrm{L}} \mright] \mright]\\
&= \nabla_{\nabla_{v}\eta^{\mathrm{L}} }\eta^{\mathrm{L}}.
\end{align*}
From here the statement follows directly from Lemma~\ref{LM10}.
\end{proof}

\begin{corollary}\label{COR11}
Suppose that the induced semi-Riemannian metric on $M$ is positive definite at $p \in M$. Then, under the hypotheses of Proposition~\ref{PROP9}, the nonzero eigenvalues of $K$ are negative and have even multiplicities.
\end{corollary}
\begin{proof}
We deduce from equations \eqref{EQ2} and \eqref{EQ3} that the invariant shape operator $\alpha$ of $G$ is skew-adjoint with respect to the semi-Riemannian metric. On the other hand, $K$ is self-adjoint. Thus, if the induced metric on $M$ is positive definite at $p$ and $K(T_{p}M) \subset T_{p}M$, then the matrices of $K$ and $\alpha$ with respect to any orthonormal basis of $T_{p}M$ are symmetric and skew-symmetric, respectively. Since $K =\alpha \circ \alpha$ when the hypotheses of Proposition~\ref{PROP9} are fulfilled, the statement follows from \cite[Theorem~2]{rinehart1960}.
\end{proof}

\begin{remark}
Dropping the assumption that the metric is positive definite, the matrix of $\alpha$ in any orthonormal basis $(b_{1}, \dotsc, b_{m})$ of $T_{p}M$ becomes
\begin{equation*}
\begin{pmatrix}
0 & -\mfrac{\langle \alpha(b_{1}), b_{2}\rangle}{\langle b_{1}, b_{1} \rangle} & \cdots & -\mfrac{\langle \alpha(b_{1}), b_{m}\rangle}{\langle b_{1}, b_{1} \rangle}\\[10pt]
\mfrac{\langle \alpha(b_{1}), b_{2}\rangle}{\langle b_{2}, b_{2} \rangle} & 0 & \cdots & -\mfrac{\langle \alpha(b_{2}), b_{m}\rangle}{\langle b_{2}, b_{2} \rangle}\\
\vdots & \vdots & \ddots & \vdots\\
\mfrac{\langle \alpha(b_{1}), b_{m}\rangle}{\langle b_{m}, b_{m} \rangle} & \mfrac{\langle \alpha(b_{2}), b_{m}\rangle}{\langle b_{m}, b_{m} \rangle} & \cdots & 0
\end{pmatrix}.
\end{equation*}
Hence, when $m=2$, the operator $\alpha \circ \alpha$ is a multiple of the identity regardless of the signature of the metric.
\end{remark}

\section{Proof of the main result} \label{PROOF}
We are now ready to prove Theorem~\ref{TH1} and Proposition~\ref{PROP4}.

\begin{proof}[Proof of Theorem~\ref{TH1}]
It follows from equation \eqref{EQ6} that
\begin{equation*}
A(e_{j})= \alpha(e_{j}) + \sum_{h=1}^{m}e_{j}(\langle N,e^{\mathrm{L}}_{h}\rangle)e_{h}.
\end{equation*}
Hence, by linearity of $K$, we have
\begin{equation*}
K(A(e_{j})) = K(\alpha(e_{j})) + \sum_{h=1}^{m} e_{j}(\langle N,e^{\mathrm{L}}_{h}\rangle)K(e_{h}),
\end{equation*}
whereas
\begin{equation*}
A(K(e_{j})) = \alpha(K(e_{j})) + \sum_{h=1}^{m} K(e_{j})(\langle N,e^{\mathrm{L}}_{h}\rangle)e_{h}.
\end{equation*}

Assume that $N_{p}M$ is closed; this way, since $G$ is equipped with a bi-invariant metric, $K$ and $\alpha$ commute by Proposition~\ref{PROP9}. It follows that $K(A(e_{j})) = A(K(e_{j}))$ exactly when
\begin{equation*}
\sum_{h=1}^{m} \lambda_{h} e_{j}(\langle N,e^{\mathrm{L}}_{h}\rangle)e_{h} = \sum_{h=1}^{m} \lambda_{j} e_{j}(\langle N,e^{\mathrm{L}}_{h}\rangle)e_{h}.
\end{equation*}
Being $(e_{1}, \dotsc, e_{m})$ a basis of $T_{p}M$, we conclude that $A$ and $K$ commute if and only if $e_{j}(\langle N, e^{\mathrm{L}}_{h}\rangle) = 0$ for all $j$ and $h$ such that $\lambda_{j} \neq \lambda_{h}$.

It remains to show that $e_{j}(\langle N, e^{\mathrm{L}}_{h} \rangle) = e_{h}(\langle N, e^{\mathrm{L}}_{j} \rangle)$ when $\lambda_{j} \neq \lambda_{h}$ and the induced metric on $M$ is positive definite at $p$. To this end, identify $\alpha$, $A$, and $K$ with their matrices in the basis $(e_{j})_{j=1}^{m}$. The first is a skew-symmetric matrix, by equation \eqref{EQ3}, whereas $A$ is symmetric and $K$ diagonal. Since $\alpha$ and $K$ commute, the $(j,h)$-entries of $\alpha K$ and $K\alpha$ are equal, and so we must have $\alpha_{jh}\lambda_{h}= \lambda_{j}\alpha_{jh}$.

Assume $\lambda_{j}\neq\lambda_{h}$. Then $\alpha_{jh}=-\alpha_{hj}=0$ and so, by equation \eqref{EQ4},
\begin{align*}
A_{jh}&= \langle \mathcal{W}(e_{j}),e_{h}\rangle,\\
A_{hj}&= \langle \mathcal{W}(e_{h}),e_{j}\rangle,
\end{align*}
implying $\langle \mathcal{W}(e_{j}),e_{h}\rangle = \langle \mathcal{W}(e_{h}),e_{j}\rangle$ by symmetry of $A$.
\end{proof}

\begin{proof}[Proof of Proposition~\ref{PROP4}]
Suppose that $M$ is a three-dimensional Riemannian submanifold with closed normal bundle, and suppose that $\alpha \neq 0$. It follows by Corollary~\ref{COR11} that $K$ has one zero eigenvalue, while the remaining two are equal. Without loss of generality, we may assume that $\lambda_{3}=0$. Since $K \neq 0$, it is clear that $\lambda_{1}=\lambda_{2}\neq0$.

Extend $\eta$ to a unit normal vector field $N$ along $M$. Then, by continuity, the multiplicity of $\lambda_{3}$ is locally constant, i.e., there exists a neighborhood $U=U(N)$ of $p$ in $M$ such that the extension of $K$ has two negative definite eigenvalues in $U$.

Assume that $A$ and $K$ commute, i.e., they share a common basis of eigenvectors. Since the $0$-eigenspace of $K$ is one-dimensional, it follows that $e_{3}$ is an eigenvector of $A$. Conversely, if $e_{3}$ is an eigenvector of $A$, then its other two eigenvectors lie in the $\lambda_{1}$-eigenspace of $K$, from which we infer that $A$ and $K$ commute.
%
\end{proof}

\appendix
\section{} \label{APPA}
Here we present a direct proof of the following obvious corollary of Theorem~\ref{TH1}.

\begin{corollary}
Suppose that $N_{p}M$ is closed. If the shape operator of $M$ with respect to $\eta$ is a multiple of the identity, then $e_{j}(\langle N, e^{\mathrm{L}}_{h} \rangle) = 0$ for all $e_{j}, e_{h}$ in different eigenspaces.
\end{corollary}

\begin{proof}
Assume the hypotheses of the corollary. Since $\alpha$ commutes with $K$, each eigenspace of $K$ is invariant under $\alpha$. Moreover, being $A$ a multiple of the identity, $\langle A(e_{j}), e_{h} \rangle =0$ for $j \neq h$, and so equation \eqref{EQ6} implies 
\begin{equation*}
\langle \alpha(e_{j}), e_{h} \rangle = \pm e_{j}(\langle N,e^{\mathrm{L}}_{h}\rangle)\quad \text{for $j \neq h$}.
\end{equation*}

Clearly, if $\lambda_{j} \neq \lambda_{h}$, then $e_{j}(\langle N,e^{\mathrm{L}}_{h}\rangle)=0$, because $\alpha(e_{j})$ and $e_{h}$ are in different (orthogonal) eigenspaces of $K$.
\end{proof}

\section*{Acknowledgments}
We thank Marcos Salvai and the anonymous referees for their constructive comments, which contributed to improve the quality of the paper.

\bibliographystyle{amsplain}
\bibliography{biblio}

\providecommand{\bysame}{\leavevmode\hbox to3em{\hrulefill}\thinspace}
\providecommand{\MR}{\relax\ifhmode\unskip\space\fi MR }
\providecommand{\MRhref}[2]{%
  \href{http://www.ams.org/mathscinet-getitem?mr=#1}{#2}
}
\providecommand{\href}[2]{#2}
\begin{thebibliography}{10}

\bibitem{berndt1991}
J\"{u}rgen Berndt, \emph{Real hypersurfaces in quaternionic space forms}, J.
  Reine Angew. Math. \textbf{419} (1991), 9--26. \MR{1116915}

\bibitem{berndt1992-2}
J\"{u}rgen Berndt and Lieven Vanhecke, \emph{Curvature-adapted submanifolds},
  Nihonkai Math. J. \textbf{3} (1992), no.~2, 177--185. \MR{1199522}

\bibitem{berndt1992}
\bysame, \emph{Two natural generalizations of locally symmetric spaces},
  Differential Geom. Appl. \textbf{2} (1992), no.~1, 57--80. \MR{1244456}

\bibitem{besse2008}
Arthur~L. Besse, \emph{Einstein manifolds}, Classics in Mathematics,
  Springer-Verlag, Berlin, 2008. \MR{2371700}

\bibitem{cartan1946}
\'{E}. Cartan, \emph{Le\c{c}ons sur la {G}\'{e}om\'{e}trie des {E}spaces de
  {R}iemann}, second ed., Gauthier-Villars, Paris, 1946. \MR{0020842}

\bibitem{cecil2015}
Thomas~E. Cecil and Patrick~J. Ryan, \emph{Geometry of hypersurfaces}, Springer
  Monographs in Mathematics, Springer, New York, 2015. \MR{3408101}

\bibitem{dajczer1980}
Marcos Dajczer and Katsumi Nomizu, \emph{On sectional curvature of indefinite
  metrics. {II}}, Math. Ann. \textbf{247} (1980), no.~3, 279--282. \MR{568993}

\bibitem{dajczer2019}
Marcos Dajczer and Ruy Tojeiro, \emph{Submanifold theory: {B}eyond an
  introduction}, Universitext, Springer, New York, 2019. \MR{3969932}

\bibitem{datri1979}
J.~E. D'Atri, \emph{Certain isoparametric families of hypersurfaces in
  symmetric spaces}, J. Differential Geometry \textbf{14} (1979), no.~1,
  21--40. \MR{577876}

\bibitem{delbarco2014}
V.~del Barco, G.~P. Ovando, and F.~Vittone, \emph{On the isometry groups of
  invariant {L}orentzian metrics on the {H}eisenberg group}, Mediterr. J. Math.
  \textbf{11} (2014), no.~1, 137--153. \MR{3160618}

\bibitem{fornari2004}
Susana Fornari and Jaime Ripoll, \emph{Killing fields, mean curvature,
  translation maps}, Illinois J. Math. \textbf{48} (2004), no.~4, 1385--1403.
  \MR{2114163}

\bibitem{graves1978}
Larry Graves and Katsumi Nomizu, \emph{On sectional curvature of indefinite
  metrics}, Math. Ann. \textbf{232} (1978), no.~3, 267--272. \MR{478082}

\bibitem{gray2004}
Alfred Gray, \emph{Tubes}, second ed., Progress in Mathematics, vol. 221,
  Birkh\"{a}user Verlag, Basel, 2004. \MR{2024928}

\bibitem{koike2005}
Naoyuki Koike, \emph{Actions of {H}ermann type and proper complex equifocal
  submanifolds}, Osaka J. Math. \textbf{42} (2005), no.~3, 599--611.
  \MR{2166724}

\bibitem{koike2014-1}
\bysame, \emph{A {C}artan type identity for isoparametric hypersurfaces in
  symmetric spaces}, Tohoku Math. J. (2) \textbf{66} (2014), no.~3, 435--454.
  \MR{3266740}

\bibitem{koike2014-2}
\bysame, \emph{The constancy of principal curvatures of curvature-adapted
  submanifolds in symmetric spaces}, Differential Geom. Appl. \textbf{35}
  (2014), 103--113. \MR{3231750}

\bibitem{lee2018}
John~M. Lee, \emph{Introduction to {R}iemannian manifolds}, second ed.,
  Graduate Texts in Mathematics, no. 176, Springer, Cham, 2018. \MR{3887684}

\bibitem{milnor1976}
John Milnor, \emph{Curvatures of left invariant metrics on {L}ie groups},
  Advances in Math. \textbf{21} (1976), no.~3, 293--329. \MR{425012}

\bibitem{murphy2012}
Thomas Murphy, \emph{Curvature-adapted submanifolds of symmetric spaces},
  Indiana Univ. Math. J. \textbf{61} (2012), no.~2, 831--847. \MR{3043597}

\bibitem{nomizu1983}
Katsumi Nomizu, \emph{Remarks on sectional curvature of an indefinite metric},
  Proc. Amer. Math. Soc. \textbf{89} (1983), no.~3, 473--476. \MR{715869}

\bibitem{oneill1983}
Barrett O'Neill, \emph{Semi-{R}iemannian geometry: {W}ith applications to
  relativity}, Pure and Applied Mathematics, no. 103, Academic Press, New York,
  1983. \MR{719023}

\bibitem{ovando2016}
G.~P. Ovando, \emph{Lie algebras with ad-invariant metrics: {A} survey-guide},
  Rend. Semin. Mat. Univ. Politec. Torino \textbf{74} (2016), no.~1, 243--268.
  \MR{3772589}

\bibitem{rinehart1960}
R.~F. Rinehart, \emph{Skew matrices as square roots}, Amer. Math. Monthly
  \textbf{67} (1960), 157--161. \MR{120247}

\bibitem{ripoll1991}
Jaime~B. Ripoll, \emph{On hypersurfaces of {L}ie groups}, Illinois J. Math.
  \textbf{35} (1991), no.~1, 47--55. \MR{1076665}

\bibitem{ripoll1993}
Jaime~B. Ripoll and Marcos Sebastiani, \emph{The generalized {G}auss map and
  applications}, Rocky Mountain J. Math. \textbf{23} (1993), no.~2, 767--780.
  \MR{1226201}

\bibitem{terng1995}
C.-L. Terng and G.~Thorbergsson, \emph{Submanifold geometry in symmetric
  spaces}, J. Differential Geom. \textbf{42} (1995), no.~3, 665--718.
  \MR{1367405}

\end{thebibliography}
\end{document}